\documentclass[12pt,reqno]{amsart}%%%%%%%%%%%%{article}
\usepackage {amssymb}
\usepackage {amsmath}
\usepackage{amsthm}
\usepackage{mathtools}
\usepackage{graphicx}
\usepackage[alphabetic]{amsrefs}
\usepackage[colorlinks, citecolor=blue]{hyperref}
\usepackage{verbatim,color,geometry}
\usepackage[all]{xy}
\newtheorem{thm}{Theorem}[section]
\newtheorem{lemma}[thm]{Lemma}
\newtheorem{prop}[thm]{Proposition}
\newtheorem{cor}[thm]{Corollary}
\newtheorem{conj}[thm]{Conjecture}

\theoremstyle{definition}
\newtheorem{defn}[thm]{Definition}

\theoremstyle{remark}
\newtheorem{rmk}[thm]{Remark}

\theoremstyle{remark}

\theoremstyle{remark}

\theoremstyle{remark}

\theoremstyle{remark}

\newcommand{\po}{\ar@{}[dr]|{\text{\pigpenfont R}}}
\newcommand{\pb}{\ar@{}[dr]|{\text{\pigpenfont J}}}

\usepackage{cite}
\usepackage{amsfonts}
\title{On the finiteness of ample models}
\author{Junpeng Jiao}
\address{Department of Mathematics, University of Utah, Salt Lake City, UT 84112, USA}
\email{jiao@math.utah.edu}
\begin{document}

\maketitle
\begin{abstract}
	In this paper, we generalize the finiteness of models theorem in \cite{BCHM} to Kawamata log terminal pairs with fixed Kodaira dimension. As a consequence, we prove that a Kawamata log terminal pair with $\mathbb{R}-$boundary has a canonical model, and it can be approximated by log pairs with $\mathbb{Q}-$boundary and the same canonical model.
\end{abstract}
\section{Introduction}
Throughout this paper, the ground field $k$ is the field of complex numbers. The purpose of this paper is to prove the following theorems on the finiteness of ample models and good minimal models.

\begin{thm}[Finiteness of Ample Models]
	\label{main}
	Let $X$ be a projective normal variety of dimension $n$. Let $V$ be a finite dimensional affine subspace of the vector space $\mathrm{WDiv}_{\mathbb{R}}(X)$ which is defined over $\mathbb{Q}$. Fix a nonnegative integer $0\leq k\leq n$. Suppose $L\subset \mathcal{L}(V)$ is a closed rational polytope, such that for any $\Delta \in L$, $(X,\Delta)$ is klt and $\kappa(X,K_X+\Delta)=k$.

	Then there are finitely many rational contractions $\pi_j:X\dashrightarrow Z_j,1\leq i\leq l$, such that if $\pi : X\dashrightarrow Z$ is the ample model of $K_X+\Delta$ for some $\mathbb{R}-$divisor $\Delta \in L$, then there is an index $1\leq j \leq l$ and an isomorphism $\xi: Z_j\rightarrow Z$ such that $\pi = \xi \circ \pi_j$. 
\end{thm}
\begin{thm}[Finiteness of Good Minimal Models]
	\label{main2}
	With the above notation, if for some $\mathbb{R}-$divisor $\Delta_0 \in \mathrm{int}(L)$, $K_X+\Delta_0$ has a good minimal model. Then for any $\Delta\in L$, $K_X+\Delta$ has a good minimal model. And there are finitely many birational contractions $\phi_j:X\dashrightarrow X_j$, $1\leq i\leq m$ such that if $\phi : X\dashrightarrow Y$ is good minimal model of $K_X+\Delta$, for some $\mathbb{R}-$divisor $\Delta \in L$, then there is an index $1\leq j \leq l$ such that $(Y,\phi_*\Delta)$ is crepant birational with $ (X_j,\phi_{j *}\Delta)$. 
\end{thm}
There are several interesting applications of these results. The first is an approximation of effective klt pair with $\mathbb{R}-$boundary.

\begin{cor}
	\label{approximation of real divisor}
	Suppose $(X,\Delta)$ is a klt pair with $\mathbb{R}-$boundary and $\kappa(X,K_X+\Delta)\geq 0$. Then for any $0<\epsilon\ll 1$, we can find finitely many $\mathbb{Q}-$divisor $\Delta_i$, $1\leq i\leq l$. Such that,
	\begin{enumerate}
		\item $\Delta$ is a convex $\mathbb{R}-$linear combination of $\Delta_i$.
		\item $||\Delta-\Delta_i||<\epsilon$.
		\item There is a $k-$dimensional normal variety $Z \cong \mathrm{Proj}\ R(X,K_X+\Delta_i)$ for every $i$.
	\end{enumerate}

\end{cor}
\vspace{0.5cm}

The idea is to use the Kodaira type Canonical bundle formula on Iitaka fibration. 
Let $f:X\rightarrow Z$ be a surjective morphism of normal varieties with connected fiber, let $\Delta$ be a $\mathbb{Q}-$divisor on $X$ such that $(X,\Delta)$ is sub klt on a neighborhood of the generic fiber $F$. Suppose that 
$$K_X+\Delta\sim_{\mathbb{Q}} f^*D$$
for some $\mathbb{Q}-$Cartier $\mathbb{Q}-$divisor $D$ on $Y$, and $ H^0(F,\mathcal{O}_F(\lfloor -\Delta^{\leq 0} |_F\rfloor ))=0$. We say that $f$ is a $klt-$trivial fibration. The Kodaira type Canonical bundle formula says that 
$$D\sim_{\mathbb{Q}} K_Y+B_Y+M_Y,$$
where $B_Y$ is called the boundary part, and $M_Y$ is the moduli part, where the moduli part only depends on $(F,\Delta|_F)$.
Much is known about the birational behaviour of such formulas: In particular, it is known that, after passing to a certain birational model $Y'$ of $Y$, the divisor $M_{Y'}$ is nef and for any higher birational model $Y''\rightarrow Y'$ the induced moduli part $M_{Y''}$ on $Y''$ is the pullback of $M_{Y'}$. We call such a variety $Y'$ the Ambro model of $f$.

Two of the main conjectures in higher dimensional birational geometry are:
\begin{conj}[\textbf{B-semiampleness Conjecture}]
	\label{conj1}
	Let $(X,\Delta)$ be a sub pair and let $f:(X,\Delta) \rightarrow Y$ be a $klt-$trivial fibration to an $n-$dimension variety $Y$. If $Y$ is an Ambro model of $f$, then $M_Y$ is semiample.
\end{conj}
\begin{conj}[\textbf{Nonvanishing}]
	\label{conj2}
	Let $(X,\Delta)$ be a klt pair with $\mathbb{Q}-$boundary. If $K_X+\Delta$ is pseudo effective, then there exists an effective $\mathbb{Q}-$divisor $D$ such that
	$K_X+\Delta\sim_{\mathbb{Q}} D$.
\end{conj}
Assume these two conjectures hold in dimension $n-1$, we prove the following statement which says that if a pseudo effective klt pair is the limit of effective pairs with small Kodaira dimension, then it is effective.
\begin{cor}
	\label{nonvanishing}
	Assume Conjecture \ref{conj1} and Conjecture \ref{conj2} hold in dimension $n-1$. Let $(X,\Delta)$ be a klt pair with $\mathbb{Q}-$boundary. Suppose $K_X+\Delta$ is pseudo effective, and there is an effective $\mathbb{Q}-$divisor $H$ on $X$, such that for any $0<\epsilon\ll 1$, we have
	$$0\leq\kappa(X,K_X+\Delta+\epsilon H)<n.$$
	Then there exists an effecitve $\mathbb{Q}-$divisor $D$ such that $K_X+\Delta\sim_{\mathbb{Q}} D$.
\end{cor}

Another interesting application concerns the MMP with scaling.
\begin{cor}
	\label{MMP with scaling}
	Let $(X,\Delta)$ be a klt pair with $\mathbb{Q}-$boundary, $\kappa(X,K_X+\Delta)\geq 0$ and $H$ is a pseudo effective $\mathbb{Q}-$divisor, such that $(X,\Delta+H)$ is klt. Suppose we can run the $K_X+\Delta$ MMP with scaling of $H$ to get a sequence $\phi_i:X_i\dashrightarrow X_{i+1}$ of $K_X+\Delta$ flips and divisorial contractions and real numbers $1\geq \lambda_1\geq \lambda_2\geq ...$ such that $K_{X_i}+\Delta_i+t H_i$ is nef for $t\in[\lambda_i,\lambda_{i+1}]$. Let $\lambda:=\lim\limits_{i\rightarrow \infty} \lambda_i$. If $K_X+\Delta +H$ has a good minimal model and $\lambda\neq \lambda_i$ for any $i\in \mathbb{N}$, then $\lambda =0$
	
\end{cor}

\begin{proof}[\textbf{\emph{Acknowledgement}}]
	I would like to thank my advisor, Professor Christopher Hacon, for many useful suggestions, discussions, and his generosity. I also thank Professor Kenta Hashizume and Jingjun Han for helpful comments and references.
\end{proof}

\section{Preliminary}
\begin{defn}
	A log pair $(X,\Delta)$ consists of a projective normal variety $X$ and an effective $\mathbb{R}-$Weil divisor $\Delta$ such that $K_X+\Delta$ is $\mathbb{R}-$Cartier. We say a pair $(X,\Delta)$ has $\mathbb{Q}$ boundary (respectively $\mathbb{R}$ boundary) if $\Delta$ is $\mathbb{Q}-$Weil(respectively $\mathbb{R}-$Weil ). We say a pair $(X,\Delta)$ is klt (respectively lc) if the discrepancies satisfy $a(E,X,\Delta)>-1$ (respectively $\geq -1$) for every prime divisor $E$ over $X$.
\end{defn}

\begin{defn}
	Let $\mathbb{K}$ denote either the rational number field $\mathbb{Q}$ or the real number field $\mathbb{R}$. Let $\pi:X\rightarrow U$ be a morphism of projective normal varieties, let $V$ be a finite dimensional affine subspace of the $\mathbb{K}-$vector space $\mathrm{WDiv}_{\mathbb{K}}(X)$ of Weil divisors on $X$. For a $\mathbb{K}$-divisor $A$, define
	\begin{enumerate}
		\item $V_A=\{\Delta\ \ |\ \Delta =A+B, B\in V\}$,
		\item $\mathcal{L}(V)=\{\Delta\in V\ |\ (X,\Delta)$ is log canonical$\}$,
		\item $\mathcal{L}_A(V)=\{\Delta=A+B \in V_A\ \ |\ (X,\Delta)$ is log canonical and $B \geq 0\}$,
		\item $\mathcal{E}(V)=\{\Delta\in \mathcal{L}(V)\ | K_X+\Delta$ is pseudo effective$\}$,
		\item And given a rational contraction $\phi:X\dashrightarrow Z$, define
		\\ $\mathcal{A}_{\pi,\phi}(V)=\{\Delta \in \mathcal{L}(V)\ |$ $Z$ is the ample model of $K_X+\Delta$ over $U\ \}$.
	\end{enumerate}
\end{defn}
For an $\mathbb{R}-$divisor $D=\sum d_iD_i$ where the $D_i$ are the irreducible components of $D$, define $||D_i|| := \max\{|d_i|\}$.
\begin{defn}
	Let $X$ be a projective normal variety. Let $V$ be a finite dimensional affine subspace of the vector space $\mathrm{WDiv}_{\mathbb{R}}(X)$. Let $ \Omega\subset \mathcal{L}(V)$ be a subset, define 
	$$\mathrm{totaldiscrep}(X,\Omega):=\inf_{E,D}\{a(E,X,D)\ |\ E\ \mathrm{is\ a\ prime\ divisor\ over\ }X,\ D\in \Omega\}$$
\end{defn}
\begin{defn}
	Let $X$ be a projective normal variety, $D$ an $\mathbb{R}-$Cartier divisor. If $D\geq 0$, define the Iitaka dimension to be
	$$\kappa(X,D)=\mathrm{max}\{ k\in \mathbb{Z}_{\geq 0}\ |\ \limsup_{m\rightarrow \infty} m^{-k}\mathrm{dim} H^0(X,\lfloor mD\rfloor)>0 \}$$
	If $D$ is not effective, define the invariant Iitaka dimension to be
	$$\kappa_\iota(X,D)=\kappa(X,E)$$
	if there is an $\mathbb{R}-$divisor $E\geq 0$ such that $E\sim_{\mathbb{R}} D$, otherwise define $\kappa_\iota(X,D)=-\infty$. It is easy to see that $\kappa_\iota(X,D)$ does not depend on the choice of $E$.
	
	let $A$ be an ample divisor on X, set
	$$\nu(X,D,A)=\mathrm{max}\{ k\in \mathbb{Z}_{\geq 0}\ |\ \limsup_{m\rightarrow \infty} m^{-k}\mathrm{dim} H^0(X,\lfloor mD\rfloor +A)>0 \}$$
	if $H^0(X,\lfloor mD\rfloor +A)\neq 0$ for infinitely many $m\in \mathbb{N}$ or $\nu(X,D,A)=-\infty$ otherwise. 
	
	Define the numerical Iitaka dimension to be
	$$\nu(X,D)=\max_{A\ \mathrm{ample}} \nu(X,D,A)$$
\end{defn}

We can also define the relative Iitaka dimension.

Let $X\rightarrow U$ be a projective morphism of normal varieties, and let $D$ be an $\mathbb{R}-$Cartier $\mathbb{R}-$divisor on $X$. Then the relative invariant Iitaka dimension of $D$, denoted by $\kappa_\iota(X/U,D)$, is defined: If there is an $\mathbb{R}-$divisor $E\geq 0$ such that $D\sim_{\mathbb{R},U} E$, set $\kappa_\iota(X/U,D)=\kappa_\iota(F,D|_F)$, where $F$ is a general fiber of the Stein factorization of $X\rightarrow U$, and otherwise we set $\kappa_\iota(X/U,D)=-\infty$. Similarly, we define the relative numerical Iitaka dimension to be $\nu(X/U,D)=\nu(F,D|_F)$. When there is $E\geq 0$ such that $D\sim_{\mathbb{R},U} E$, it is easy to see that $\kappa_\iota(X/U,D)$ and $\nu(X/U,D)$ do not depend on the choice of $E$ and $F$.
\begin{prop}
	\label{prop}
	Let $X\rightarrow U$ be a projective morphism of normal varieties and $D$ be a pseudo-effective $\mathbb{R}-$divisor over $U$.
	\begin{enumerate}
		\item[(1).] If $D_1\geq 0,D_2\geq 0$ are two $\mathbb{R}-$divisors, then
		$$\nu(X/U,D_1+D_2)\geq \mathrm{max}\{\nu(X/U,D_1),\nu(X/U,D_2)\}.$$
		\item[(2).] If $D\geq 0$ is an $\mathbb{R}-$divisor, then $\nu(X/U,D)\geq \kappa(X/U,D)$.
		\item[(3).] If $D'$ is an $\mathbb{R}-$divisor with $D'-D$ being pseudo-effective, then $\nu(X,D')\geq \nu(X,D).$
		\item[(4).] Suppose that $D_1\sim_{\mathbb{R},U} N_1$ and $D_2\sim_{\mathbb{R},U} N_2$ for some $\mathbb{R}-$divisors $N_1\geq 0$ and $N_2\geq 0$ such that $\mathrm{Supp}\ N_1\subseteq\mathrm{Supp}\ N_2$. Then we have $\kappa_\iota(X/U,D_1)\leq\kappa_\iota (X/U,D_2)$ and $\nu(X/U,D_1)\leq\nu(X/U,D_2)$.
	\end{enumerate}
	\begin{proof}
		(1), (2) are obvious, (3) comes from \cite[Proposition 5.2.7]{ZDA}, (4) comes from \cite[Remark 2.8]{hashizume2019minimal}.
	\end{proof}
\end{prop}
\begin{lemma}
	\label{vandk}
	Let $(X,\Delta)$ be a klt pair with $\mathbb{R}-$boundary. Then $(X,\Delta)$ has a good minimal model if and only if $\kappa_\iota(X,K_X+\Delta)=\nu(X,K_X+\Delta)$.
	\begin{proof}
		This is \cite[Lemma 2.13]{hashizume2019minimal}.
	\end{proof}
\end{lemma}
\begin{lemma}
	\label{good minimal model}
	Let $f:X\rightarrow U$ be a projective morphism, $(X,\Delta)$ a dlt pair with $\mathbb{Q}$ boundary and $\phi:X\dashrightarrow X_M$ and $\phi' :X\dashrightarrow X'_M$ be minimal models for $K_X+\Delta$ over $U$. Then
	\begin{enumerate}
		\item the set of $\phi-$exceptional divisors coincides with the set of divisors contained in $\mathbf{B}_-(K_X+\Delta/ U)$ and if $\phi$ is a good minimal model for $K_X+\Delta$ over $U$, then this set also coincides with the set of divisors contained in $ \mathbf{B}(K_X+\Delta / U)$.
		\item $X'_M\dashrightarrow X_M$ is an isomorphism in codimension 1 such that $a(E;X_M,\phi_*\Delta)=a(E;X'_M,\phi'_*\Delta)$ for any divisor $E$ over $X$, and
		\item if $\phi $ is a good minimal model of $K_X+\Delta$ over $U$, then so is $\phi '$.
		
	\end{enumerate}
	\begin{proof}
		This is \cite[Lemma 2.4]{HMX}.
	\end{proof}
\end{lemma}
\begin{thm}
\label{thm3}
	Let $X$ be a projective normal variety. If $(X,\Delta)$ is a klt pair with $\mathbb{Q}-$boundary and $\kappa(X,\Delta)\geq 0$, then the ample model of $(X,\Delta)$ exists.
	\begin{proof}
		This is \cite[Corollary 1.1.2]{BCHM}.
	\end{proof}
\end{thm}
\section{Proof of Main Theore}
In this section we prove the relative version of Theorem \ref{main} and Theorem \ref{main2}.

\begin{lemma}
	\label{thm2}
	Let $X\rightarrow U$ be a projective morphism of normal varieties, fix an integer $0\leq k\leq n$. Let $V$ be a finite dimensional affine subspace of the vector space $\mathrm{WDiv}_{\mathbb{R}}(X)$ of Weil divisors on $X$ which is defined over the rationals. Suppose $L$ is a convex subset, such that for any $\mathbb{Q}-$divisor $D\in L$, $\kappa(X/U,D)=k$ and the ample model of $D$ over $U$ exists. Suppose $f_D:X\dashrightarrow Z_D$ is the ample model of $D$ over $U$. Then the generic fiber of $f_D$ are the same for every $D\in L$, in particular, $Z_D$ are birational equivalent for every $D\in L$.
	\begin{proof}
		Choose two $\mathbb{Q}-$divisors $D_1,D_2\in L$. After replaceing $X$ by a higher model, and $D_i$ by its pull back, we may assume that $X\rightarrow Z_{D_i}$ is a morphism. Then we have the diagram
		$$\xymatrix{
			&X \ar[dl] _{f_1} \ar[dr]^{f_2} \ar[d]_{h}&  \\
			Z_{D_1}	& Z_{D_1} \times Z_{D_2} \ar[l]^{p_1} \ar[r]_{p_2} & 	 Z_{D_2}
		} $$
		where $p_1$ and $p_2$ are two projections from the fiber product $Z_{D_1} \times Z_{D_2}$. Since $Z_{D_i}$ is the ample model of $D_i$ over $U$, there exists an ample divisor $A_i$ on $Z_{D_i}$, such that $$D_i\sim_{\mathbb{Q},U} f_i ^* A_i +E_i,$$ where $E_i$ is effective, $i=1,2$.
		
		By the definition of fiber product, $A:=\frac{1}{2}(p_1 ^* A_1+p_2^* A_2)$ is an ample divisor on $Z_{D_1} \times Z_{D_2}$. Let $Z$ be the normalization of image of $X$ in $Z_{D_1}\times Z_{D_2}$. We have $A|_Z$ is ample on $Z$, in particular, 
		$$\kappa(Z/U,A|_Z)=\mathrm{dim}(Z)-\mathrm{dim}(U)\geq \mathrm{dim}(Z_{D_i})-\mathrm{dim}(U)=\kappa(X/U,D_i)=k,$$ 
		Because 
		$$\frac{1}{2}(D_1+D_2)\sim_{\mathbb{Q},U} h^* A|_Z+\frac{1}{2}(E_1+E_2),$$ 
		and $\frac{1}{2}(D_1+D_2)\in L$, we have that 
		$$k=\kappa(X/U,\frac{1}{2}(D_1+D_2))\geq \kappa(X/U,h^*A|_Z)=\kappa(Z/U,A|_Z)\geq k,$$
		which means $\mathrm{dim}\ Z=\mathrm{dim}\ Z_{D_i}$. Because all $f_i$ are algebraic contractions, it is easy to see that the morphisms $p_i:Z\rightarrow Z_i$ are birational. 
	\end{proof}
\end{lemma}

\subsection{Finiteness of Ample Models}
	
	Let $h:X\rightarrow Y$ be an equidimensional algebraic fibration over $U$ such that $Y$ is smooth. Suppose $(X,\Delta)$ is a klt pair with $\mathbb{Q}-$boundary,  $\kappa(X/U,K_X+\Delta)=k\geq 0$, and the restriction $h_\eta:X_\eta\rightarrow Y_\eta$ over the generic point $\eta$ is birational to the Iitaka fibration of $K_X+\Delta$ over $U$. First we show how to get a pair $(Y,C)$ of log general type from $(X,\Delta)$.
	
	Since $\kappa(X/U,K_X+\Delta)=k\geq 0$, there is a $\mathbb{Q}-$effective divisor $L$ such that 
	$$K_X+\Delta\sim_{\mathbb{Q},U} L.$$
	We put $D := \mathrm{max}\{N\ |\ N$ is an effective $\mathbb{Q}$-divisor on $Y$ such that $L \geq h^* N \}$ and $F:=L-h^*D$. Then we have
	$$K_X+\Delta\sim_{\mathbb{Q},U} h^*D+F$$
	By definition, $\kappa(X_\eta,(K_X+\Delta)|_\eta)=0$, so $f_*\mathcal{O}_X(\lfloor iF\rfloor)$ is a reflexive sheaf of rank 1 on $Y$. Moreover, since $Y$ is smooth, $h_*\mathcal{O}_X(\lfloor iF\rfloor)$ is an invertible sheaf on $Y$. By construction, $\mathrm{Supp}(F)$ does not contain the whole fiber of any prime divisor on $Y$, therefore we have $\mathcal{O}_Y\cong h_*\mathcal{O}_X(\lfloor iF\rfloor)$. Moreover, it is easy to see that $D$ and $F$ are both $\mathbb{Q}-$divisors.
	\begin{rmk}
		\label{piecewise linear}
		We will show how $L$ and $F$ vary depending on $\Delta$. Define
		$$\mathcal{D}=\{(a_1,a_2,...,a_m)\in [0,1] ^{\times m}\ | \sum_{i=1}^{m}a_i=1\}$$
		
		Let $h:X\rightarrow Y$ be an equidimensional algebraic fibration and $Y$ is smooth, let $X_\eta$ denote the generic fiber of $h$. Suppose $\{L_i,1\leq i\leq m\}$ are $m$ linearly independent effective $\mathbb{Q}-$Cartier $\mathbb{Q}-$divisors, such that for every $(a_1,a_2,...,a_m)\in \mathcal{D}$, we have $\kappa(X_\eta, \sum_{i=1}^{m}a_iL_i)=0$. Define 
		$$D(a_1,...,a_m) := \mathrm{max}\{N\ |\ N\ \mathrm{is\ an\ effective\ }\mathbb{Q}-\mathrm{divisor\ on\ }Y\ \mathrm{such\ that\ }\sum_{i=1}^{m}a_iL_i \geq f^* N \}.$$ 
		and 
		$$F(a_1,...,a_m):=\sum_{i=1}^{m}a_iL_i-f^*D(a_1,...,a_m).$$
		
		Next we show that $D(a_1,...,a_m)$ is a piecewise $\mathbb{Q}-$linear function on $\mathcal{D}$. Let $P$ be a prime divisor on $Y$, suppose $h^*P=\sum_{j=1}^{l}b_j G_j$, where $G_j,1\leq j\leq l$ are prime divisors. Then the coefficient of $P$ in $D(a_1,...,a_m)$ is 
		$$\mathrm{coeff}_P D(a_1,...,a_m)= \mathrm{min}\{\sum_{i=1}^{m}\frac{a_i}{b_j} \mathrm{coeff}_{G_j} L_i,\ 1\leq j\leq m\}.$$
		It is easy to see that $\mathrm{coeff}_P D(a_1,...,a_m)$ is a piecewise linear function of $(a_1,...,a_m)$. Moreover, because there are only finitely many prime divisors $P$ such that $\mathrm{Supp}\ h^*P\subset \cup_{1\leq i\leq m}\mathrm{Supp}\ L_i$, and $\frac{1}{b_j} \mathrm{coeff}_{G_j}L_i$ are rational numbers for every $j$, we can divide $\mathcal{D}$ into finitely many rational polytopes $\cup_k \mathcal{D}_k$ such that $D(a_1,...,a_m)$ is a $\mathbb{Q}-$linear function in each $\mathcal{D}_k$.
	\end{rmk}
	\begin{thm}
		\label{fibration}
		Let $\pi:X\rightarrow U$ be a projective morphism of normal varieties and $\mathrm{dim}(X)=n$. Let $V$ be a finite dimensional affine subspace of the vector space $\mathrm{WDiv}_{\mathbb{R}}(X)$ which is defined over the rationals. Fix an integer $0\leq k \leq n$. Suppose $L\in \mathcal{L}(V)$ be a closed convex rational polytope, such that for any $\Delta \in L$, $(X,\Delta)$ is klt and $\kappa(X/U,K_X+\Delta)=k$. Then there exists a commutative diagram
		$$\xymatrix{
			X  \ar[d] _{\pi}& &   X' \ar[ll]_{\mu} \ar[d]^{h}\\
			U	&  & 	 Y\ar[ll]^{g}
		} $$
		with the following properties.
		\begin{enumerate}
			\item $\mu$ is birational morphism, $h$ is an equidimensional algebraic fibration, $X'$ has only $\mathbb{Q}-$factorial toroidal singularities and $Y$ is smooth;

			\item There exits a finite dimensional affine subspace $V'$ of the vector space $\mathrm{WDiv}_{\mathbb{R}}(X')$ which is defined over $\mathbb{Q}$, a closed rational polytope $L'\subset \mathcal{L}(V')$ with $\mathrm{totaldiscrep}(X',L')= \mathrm{totaldiscrep}(X,L)$, and a $\mathbb{Q}-$linear isomorphism $*': L\rightarrow L'$. For any divisor $\Delta\in L$, $(X',\mathrm{Supp}(\Delta'))$ is quasi-smooth ($i.e.$, $(X',\mathrm{Supp}(\Delta'))$ is toriodal), and 
			$$\mu _*\mathcal{O}_{X'}(m(K_{X'}+\Delta')) \cong \mathcal{O}_X(m(K_X+\Delta)),\ \ \forall m\in \mathbb{N};$$
			
		\end{enumerate}
		\begin{proof}
			Fix $\Delta_1\in L$, we may choose a birational projective morphism $\mu:X'\rightarrow X$, such that there exists a projective morphism $h:X'\rightarrow Y$ of smooth projective varieties over $U$ and the restriction $h_\eta: X'_\eta\rightarrow Y_\eta$ over the generic point $\eta$ of $U$ is birational to the Iitaka fibration of $K_X+\Delta_1$. By Lemma \ref{thm2}, it is birational to the Iitaka fibration of $K_X+\Delta$, for every $\Delta \in L$. Let $\Delta_0\in L$ be an inner rational point. By the weak semi-stable reduction therorem of Abramovich of Karu (cf. \cite{Abramovich2000}), we can assume that, $h:(X',D')\rightarrow (Y,D_Y)$ is an equidimensional toroidal morphism for some divisors $D'$ on $X'$ and $D_Y$ on $Y$ where $(X',D')$ is quasi-smooth, $Y$ is smooth and $\mu ^{-1}(\Delta_0 \cup \mathrm{Sing}(X)) \subset D'$.  
			
			Let $F$ be the exceptional divisor of $\mu$, denote $a:=\mathrm{totaldiscrep}(X,L)$. For a $\Delta \in L$, define $ \Delta'$ by 
			$$\Delta' :=\mu_* ^{-1}\Delta +\max\{0,-a\} F$$
			Since $\Delta_0$ is an inner point of $L$, it is easy to see that $\mathrm{Supp}(\Delta')\in D'$, clearly $(X',\Delta')$ satisfies (1) and (2).
		\end{proof}
	\end{thm}
	
	The following theorem is a relative version of Theorem \ref{main}.
	\begin{thm}
		\label{relative ample model}
		Let $\pi:X\rightarrow U$ be a projective morphism of normal varieties and $\mathrm{dim}(X)=n$. Let $V$ be a finite dimensional affine subspace of the vector space $\mathrm{WDiv}_{\mathbb{R}}(X)$ which is defined over $\mathbb{Q}$. Fix a nonnegative integer $0\leq k\leq n$. Suppose $L\subset \mathcal{L}(V)$ is a closed rational polytope, such that For any $\Delta \in L$, $(X,\Delta)$ is klt and $\kappa(X,K_X+\Delta)=k$.
		
		Then there are finitely many rational contractions $\pi_j:X\dashrightarrow Z_j,1\leq i\leq l$ over $U$, such that if $\pi : X\dashrightarrow Z$ is an ample model of $K_X+\Delta$ over $U$ for some $\mathbb{Q}-$divisor $\Delta \in L$, then there is an index $1\leq j \leq l$ and an isomorphism $\xi: Z_j\rightarrow Z$ such that $\pi = \xi \circ \pi_j$. 
		\begin{proof}
			It is easy to see that we may assume that $L$ is a convex polytope. By Theorem \ref{fibration}, there is an equidimensional algebraic fibration $h:X'\rightarrow Y$ over $U$ and a $\mathbb{Q}-$linear isomorphism $*': L \rightarrow L'$ with $\mathrm{totaldiscrep}(X,L')=\mathrm{totaldiscrep}(X,L)$. Moreover, $Z$ is the ample model of $(X,\Delta)$ over $U$ if and only if it is the ample model of $(X',\Delta')$ over $U$. Therefore we can replace $X$ by $X'$ and $L$ by $L'$. Then we can assume that there is an equidimensional algebraic fibration $h:X\rightarrow Y$ over $U$, $Y$ is smooth and for every $\mathbb{Q}-$divisor $\Delta\in L$, $\kappa(X_\eta,(K_X+\Delta)|_\eta)=0$, where $X_\eta$ is the generic fiber of $h$.
			
			By Remark \ref{piecewise linear}, we may divide $L$ into finitely many simplexes $\cup_i L_k$. For every $\Delta\in L$, there are divisors $D_\Delta$ and $F_\Delta$ satisfying the following equations
			$$K_X+\Delta\sim_{\mathbb{Q},U} h^* D_\Delta +F_\Delta$$
			$$\mathcal{O}_Y\cong h_*\mathcal{O}_X(\lfloor jF\rfloor)$$
			for all $j\geq 0$. Moreover, $D_\Delta$ is a linear function of $\Delta$ in each $L_k$. Because this division is finite and rational, we only need to prove finiteness of ample models for such $L_k$. Let $\{\Delta_i,1\leq i\leq m\}$ be the vertexes of $L_k$ and let $D_i, F_i$ be short for $D_{\Delta_i}$ and $F_{\Delta_i}$.
			
			It follows from the proof of \cite{Kollr2007} that for every $\mathbb{Q}-$divisor $\Delta_i,1\leq i\leq m$, there exist $\mathbb{Q}-$divisors $B_i$ and $J_i$ on $Y$, such that $(Y,B_Y)$ is klt and
			$$D_i\sim_{\mathbb{Q},U} K_Y+B_i+J_i.$$
			Since the $\mathbb{Q}-$line bundel $J_i$ commutes with birational pull back. We may find a birational morphism $f:Y'\rightarrow Y$, such that
			$$K_{Y'}+B'_i+J'_i\sim_{\mathbb{Q}} f^*(K_Y+B_i+J_i).$$
			where $J'_i$ is a $U-$nef $\mathbb{Q}-$line bundle, $(Y',B'_i)$ is subklt and $f_*B'_i=B_i$. Because $D_i$ is big, there is an ample $\mathbb{Q}-$line bundle $A'$ and big $\mathbb{Q}-$divisors $E_i$ on $Y'$, such that $f^*D_i\sim_{\mathbb{Q}} f^*A'+E_i$. Let $\epsilon>0$ be a small enough rational number such that $(Y',B'_i+C'_i+\epsilon f^*A'),1\leq i\leq m$ is sub klt, where $C'_i\in |J'_i+\epsilon E_i|_{\mathbb{Q}}$ is a general member. Let $C_i:=f_*(B'_i+C'_i)$ and $A\in|\epsilon A'|_{\mathbb{Q}}$ be a general member, then we have $(Y,A+C_i),1\leq i\leq m$ is klt and 
			$$(1+\epsilon)(K_Y+B_i+J_i)\sim_{\mathbb{Q}} K_Y+A+C_i.$$ 
			Let $W\subset \mathrm{WDiv}_{\mathbb{Q}}(Y)$ be the finite dimensional subspace spanned by the $\{C_i,1\leq i\leq m\}$. For a point $(a_1,...,a_m)\in \mathcal{D}$, we have 
			$$(1+\epsilon)(K_X+\sum_{i=1}^{m}a_i\Delta_i)\sim_{\mathbb{Q},U} h^*(K_Y+A+\sum_{i=1}^{m}a_iC_i)+R'_{(a_1,...,a_m)}$$
			where $R'_{(a_1,...,a_m)}$ is $\mathbb{Q}-$effective and $\mathcal{O}_Y\cong h_*\mathcal{O}_X(\lfloor jR'_{(a_1,...,a_m)}\rfloor)$ for all $j\geq 0$.
			Therefore a projective variety $Y'$ is the ample model of $(X,\sum_{i=1}^{m}a_i\Delta_i)$ if and only if it is the canonical model of $(Y,A+\sum_{i=1}^{m}a_iC_i)$. Then we have a $\mathbb{Q}-$linear map $C_*:L_k\rightarrow\mathcal{L}_A(W)$, and for every $\mathbb{R}-$divisor $\Delta \in L_k$, $(X,\Delta)$ and $(Y,C_\Delta) $ have the same ample model over $U$. Therefore the claim is now immediate from \cite{BCHM} Corollary 1.1.5.
		\end{proof}
	\end{thm}
	\begin{rmk}
		\label{rmk of ample model}
		Because the map $\Delta\rightarrow D_\Delta \rightarrow C_\Delta$ is $\mathbb{Q}-$linear, by Theorem 
		\ref{relative ample model}, it is easy to see that if $L$ is a rational polytope, then $\mathcal{A}_{\pi,\phi}(V)\cap L$ is a rational polytope.
	\end{rmk}

	\subsection{Finiteness of good minimal models}

	\begin{lemma}
		\label{numerical}
		Let $\pi:X\rightarrow U$ be a projective morphism of normal varieties, $V$ be a finite dimensional affine subspace of $\mathrm{WDiv}_{\mathbb{R}}(X)$ which is defined over the rationals. Fix a nonnegative integer $0\leq k\leq n$, let $L\subset V$ be a closed convex rational polytope of $\mathbb{R}-$divisors, such that for any $\mathbb{Q}-$divisor $D\in \Omega$, $\kappa(X/U,D)=k\geq 0$.
		
		If there is a $\mathbb{R}-$divisor $D_0\in \mathrm{int}(L)$ satisfying $\nu(X/U,D_0)=\kappa_\iota(X/U,D_0)$. Then for any $\mathbb{R}-$divisor $D\in L$, we have $\nu(X/U,D)=\kappa_\iota(X/U,D)$.
		\begin{proof}
			Let $D_i,1\leq i\leq m$ be the vertexes of $L$, by assumption, $D_i$ are $\mathbb{Q}-$divisors and $\kappa(X/U,D_i)=k$. Choose $E_i$ such that $D_i\sim_{\mathbb{Q},U}E_i\geq 0$, then for every point $(a_1,a_2,...,a_m)\in \mathcal{D}$, we have $\sum_{i=1}^{m}a_iD_i\sim_{\mathbb{R},U} \sum_{i=1}^{m} a_iE_i\geq 0$. By Proposition \ref{prop}, we may replace $D_i$ by $E_i$, then we may assume $D\geq 0$ for every $D\in L$ and replace $\kappa_\iota(X,D)$ by $\kappa(X,D)$.
			
			If for some $D_0\in \mathrm{int}(L)$, we have $\nu(X/U,D)=\kappa(X/U,D)$. Then for every $D\in L$, we have $\mathrm{Supp}(D)\subset \mathrm{Supp}(D_0)$. Thus the claim comes easily from the following inequality.
			$$\kappa(X/U,D)\leq\nu(X/U,D)\leq\nu(X/U,D_0)=\nu(X/U,D_0)=\kappa(X/U,D_0)=k=\kappa(X/U,D).$$
		\end{proof}
	\end{lemma}
	
	\begin{thm}
		\label{countably}
		Let $X$ be a projective variety, then there are at most countably many birational contractions $X\dashrightarrow X_i, i=1,2,...$.
		\begin{proof}
			If $f:X\dashrightarrow X'$ is a birational contraction, let $p:W\rightarrow X$ and $q:W\rightarrow X'$ resolve the indeterminacy of $f$. Let $A$ be a general ample divisor on $X'$, define $D:=p_*q^*A$. By negativity lemma, we have $p^*D=q^*A+E$ for some $p-$exceptional divisor $E\geq 0$, hence $q-$exceptional. It is easy to see that 
			$$R(X,D)=R(W,p^*D)=R(X',A).$$
			If $D'$ is divisor on $X$ satisfying $D'\equiv D $, then $q_*p^* D'\equiv q_*p^*D=A$, and $q_*p^* D' $ is an ample divisor on $X'$, denote it by $A'$. Since $p^*D'-q^*A' \equiv p^*D-q^*A$ and they are both $q-$exceptional, by the negativity lemma, we have $p^* D'-q^*A'=p^*D-q^*A=E\geq 0$, hence $ X'=\mathrm{Proj}\ R(X',A')=\mathrm{Proj}\ R(X,D')$. Therefore each birational contraction is determined by the numerical class of a big $\mathbb{Q}-$divisor. Since $N^1(X)$ is finite dimensional, the claim follows.
		\end{proof}
	\end{thm}

	\begin{thm}
		
		\label{finiteness of good minimal model}
		Let $X$ be a projective normal variety over a normal variety $U$, $V$ be an affine finite dimensional subspace of $\mathrm{WDIV}_{\mathbb{R}}(X)$ which is defined over the rationals, suppose $L$ is a rational polytope of $\mathcal{L}(V)$ such that for any $\mathbb{Q}-$divisor $\Delta\in L$, $(X,\Delta)$ has a good minimal model over $U$ and has the same ample model $Z$ over $U$. Then there is a birational contraction $f:X\dashrightarrow X'$, such that for any $\mathbb{R}-$divisor $\Delta\in L$, $f$ is a good minimal model of $(X,\Delta)$ over $U$.
		\begin{proof}
			Because for any $\mathbb{Q}-$divisor $\Delta\in L$, $(X,\Delta)$ have the same ample model $Z$ over $U$ and $(X,\Delta)$ has a good minimal model, it follows that $\kappa(X/U,K_X+\Delta)$ is the same, and by Lemma \ref{vandk} and Lemma \ref{numerical},  for any $\mathbb{R}-$divisor $\Delta$, we have $\kappa_\iota(X,K_X+\Delta)=\nu(X,K_X+\Delta)$, which implies that for any $\mathbb{R}-$divisor $\Delta \in L$, $ K_X+\Delta$ has a good minimal model over $U$. 
			
			By Theorem \ref{countably}, $X$ has at most countably many birational contractions $f:X\dashrightarrow X_i,\ i\geq 0$. Thus we can divide $L$ into countably many subset $\cup_{i\geq 0} L_i$, such that for any $\mathbb{R}-$divisor $\Delta \in L_i$, $ f_i$ is a good minimal model of $(X,\Delta)$ over $U$. Because $f:X\dashrightarrow X'$ is a good minimal model for both $ (X,\Delta_1)$ and $(X,\Delta_2)$ implies that $f$ is also a good minimal model for  $(X,\lambda \Delta_1+(1-\lambda)\Delta_2),\ \forall \lambda\in[0,1]$, then we have each $L_i$ is convex.
			
			By the Pigeon-hole Principle, we may assume $L_1$ spans $V$, which means there are some $\mathbb{Q}-$divisors $\{\Delta_i\in L_1,\ 1\leq i\leq m\}$ spanning $V$ and $f_1$ is a good minimal model for $(X,\Delta_j),\ 1\leq j\leq m$. So $K_{X_1}+f_{1*}\Delta_j, 1\leq j\leq m$ are semiample over $U$. Because $(X,\Delta_j)$ and $(X_1,f_{1*}\Delta_j)$ have the same ample model $Z$ for every $1\leq j\leq m$, by \cite[Lemma 3.6.5]{BCHM}, $K_{X_1}+f_{1*}\Delta_j, 1\leq j\leq m$ define a morphism $h_1:X_1\rightarrow Z$, then there are some divisors $D_i,1\leq j\leq m $ on $Z$, such that
			$$K_{X_1}+f_{1*}\Delta_j \sim_{\mathbb{Q},U} h_1 ^* D_j,\ \ 1\leq j \leq m.$$
			Because $ \{\Delta_j,1\leq j\leq m\}$ span $V$, then for any divisor $\Delta\in L$, there is a divisor $D$ on $Z$ such that $K_{X_1}+f_{1*}\Delta\sim_{\mathbb{Q},U} h_1^* D$.
			
			For a $\mathbb{Q}-$divisor $\Delta\in L$, let $(X',\Delta')$ be a good minimal model of $(X,\Delta)$ and $h':X'\rightarrow Z$ be the morphism from the good minimal model to the ample model. let $\Delta^1$ be the image of $\Delta$ on $X_1$. Choose a common resolution $p:W\rightarrow X_1$ and $ q:W\rightarrow X'$ of $X'$ and $X_1$. Because $ \Delta'$ and $\Delta^1$ are the images of $\Delta$ on $X'$ and $X_1$, we have an equation 
			\begin{equation}
				\label{equa5}
				p^*(K_{X_1}+\Delta^1)+E= q^*(K_{X'}+\Delta')+F
			\end{equation}
			where $E$ is $p-$exceptional and $F$ is $q-$exceptional. By assumption, there are divisors $C'$ and $C^1$ on $Z$, such that $K_{X_1}+\Delta^1\sim_{\mathbb{Q}} h_1 ^*C^1 $ and $ K_{X'}+\Delta'\sim_{\mathbb{Q}} {h'} ^*C'$. Let $f:= p\circ h_1=q\circ h'$, then \eqref{equa5} is equal to
			$$ f^*(C_1-C_2)=F-E$$
			Since $F-E$ is either exceptional for $p$ or exceptional for $q$, $F-E$ does not contain the whole fiber of any prime divisor on $Z$. Therefore $C_1=C_2$, which means $(X',\Delta')$ is crepant birational with $(X_1,\Delta^1)$. By Lemma \ref{good minimal model} $(X_1,\Delta^1)$ is a good minimal model for $(X,\Delta)$. Therefore, for every $\mathbb{Q}-$divisor $\Delta\in L$, $f_1$ is a good minimal model for $(X,\Delta)$. Because $L$ is a rational polytope, any $\mathbb{R}-$divisor $\Delta \in L$ is a convex combination of $\mathbb{Q}-$divisors in $L$. Therefore, for any $\mathbb{R}-$divisor $\Delta \in L$, $f_1$ is a good minimal model for $(X,\Delta)$.
		\end{proof}
	\end{thm}
	The following theorem is the relative version of theorem \ref{main2}.
	\begin{thm}
		\label{relative good minimal model}
		Let $X$ be a projective normal variety over a normal variety $U$ and $\mathrm{dim}(X)=n$. Let $V$ be a finite dimensional affine subspace of the vector space $\mathrm{WDiv}_{\mathbb{R}}(X)$ which is defined over $\mathbb{Q}$. Fix a nonnegative integer $0\leq k\leq n$. Suppose $L\subset \mathcal{L}(V)$ is a closed rational polytope, such that For any $\Delta \in L$, $(X,\Delta)$ is klt and $\kappa(X,K_X+\Delta)=k$.
		
		If for some $\mathbb{R}-$divisor $\Delta_0 \in \mathrm{int}(L)$, $K_X+\Delta_0$ has a good minimal model over $U$. Then for any $\Delta\in L$, $K_X+\Delta$ has a good minimal model over $U$. And there are finitely many birational contractions $\phi_j:X\dashrightarrow X_j$, $1\leq i\leq m$, such that for any $\mathbb{R}-$divisor $\Delta\in L$, if $\phi : X\dashrightarrow Y$ is good minimal model of $K_X+\Delta$ over $U$, then there is an index $1\leq j \leq l$ such that $(Y,\phi_*\Delta)$ is crepant birational with $ (X_j,\phi_{j *}\Delta)$. 
		\begin{proof}
			If for some $\mathbb{R}-$divisor $\Delta_0 \in \mathrm{int}(L)$, $K_X+\Delta_0$ has a good minimal model over $U$, then by Lemma \ref{numerical} and Lemma \ref{vandk}, for any $\mathbb{R}-$divisor $\Delta \in U$, $K_X+\Delta$ has a good minimal model over $U$. Then the claims comes easily from Theorem \ref{relative ample model} Theorem \ref{finiteness of good minimal model}.
		\end{proof}
	\end{thm}

 	\section{Applications}
 		\subsection{Approximation of pair with $\mathbb{R}-$boundary}
 		
 	\begin{proof}[Proof of Corollary \ref{approximation of real divisor}]
 		If $\Delta$ is a $\mathbb{Q}-$divisor, then the result is straight forward. So we may assume it is not a $\mathbb{Q}-$divisor.
 		
 		First we show there is an effective $\mathbb{Q}-$divisor $B$, such that $\Delta-B\geq 0$, and $\kappa(X,K_X+B)\geq 0$.
 		
 		Since $\kappa(X,K_X+\Delta)=k\geq 0$. There is a positive integer $m$ sufficiently divisible, such that
 		$$h^0(X,\mathcal{O}_X(mK_X+\lfloor m\Delta \rfloor))>0.$$
 		So we can choose $B:=\frac{1}{m} \lfloor m\Delta \rfloor$.
 		
 		Now we may write it as an $\mathbb{R}-$linear combination of $\mathbb{Q}-$divisor
 		$$\Delta=B+\sum_{j=1}^{m} \alpha_j E_j.$$
 		Let $V$ be the space spanned by $E_j,1\leq j\leq m$. Then it is easy to see that
 		$\Delta \subset \mathrm{int}(\mathcal{L}_B(V))$.
 		
 		Because $\kappa(X,K_X+B)\geq 0$, therefore by Proposition \ref{prop}, for any $\mathbb{R}-$divisor $D\in\mathrm{int}(\mathcal{L}_B(V))$, we have $\kappa_\iota(X,K_X+D)=k$. Choose a rational polytope $L\subset \mathrm{int}(\mathcal{L}_B(V))$ that contains $\Delta$, then by Theorem \ref{main}, we may assume that for every $\mathbb{Q}-$divisor $D\in L$, $K_X+D$ has the same ample model.
 		
 		Then by \cite[Lemma 3.7.7]{BCHM}. We can find $\mathbb{Q}-$divisors $D_i\in L,1\leq i\leq l$, such that $||\Delta-D_i||\leq \epsilon$ and $\Delta$ is a convex linear combination of $D_i$. Let $Z$ be the ample model of $K_X+D_1$, then $Z\cong \mathrm{Proj}\ R(X,D_i)$ for any $D_i,1\leq i\leq l$.
 	\end{proof}
 \subsection{Nonvanishing}

\begin{proof}[Proof of Corollary \ref{nonvanishing}]
	Because $H$ is $\mathbb{Q}-$effective, $\kappa(X,K_X+\Delta+tH)$ is a nondecreasing function of $t$. Also by Proposition \ref{prop}, for $0<t_1<t_2<t_3$, we have $\kappa(X,K_X+\Delta+t_2H)>\max\{\kappa(X,K_X+\Delta+t_1H),\kappa(X,K_X+\Delta+t_3H)\}$. Thus there is an integer $0\leq k<n$, such that for $t>0$,
	$\kappa(X,K_X+\Delta+tH)=k$.

	If $k=0$, then $\kappa(X,K_X+\Delta+\epsilon H)=0$ and $h^0(X,\mathcal{O}_X(m(K_X+\Delta+\epsilon H)))\leq 1$ for $\epsilon$ small enough. Choose $\epsilon_0$ small enough, there exists an integer $m_0$ and a rational function $f_0$ on $X$, such that,
	$$\mathrm{div}(f_0)+m_0(K_X+\Delta+\epsilon_0 H)\geq 0,$$ 
	Similarly, for any rational number $0<\epsilon<\epsilon_0 $, we can find $f_\epsilon$ and $m_\epsilon$, such that
	$$\mathrm{div}(f_\epsilon)+m_\epsilon(K_X+\Delta+\epsilon H)\geq 0,$$ 
	Since $H\geq 0$, it is easy to see that $\mathrm{div}(f_\epsilon^{m_0})+m_0m_\epsilon(K_X+\Delta+\epsilon_0 H)\geq 0$ and $\mathrm{div}(f_0^{m_\epsilon})+m_0m_\epsilon(K_X+\Delta+\epsilon_0 H)\geq 0$.
	 
	Since $h^0(X,\mathcal{O}_X(m_0m_\epsilon(K_X+\Delta+\epsilon H)))\leq 1$, we have 
	$$\mathrm{div}(f_\epsilon^{m_0})+m_0m_\epsilon(K_X+\Delta+\epsilon_0 H)=\mathrm{div}(f_0^{m_\epsilon})+m_0m_\epsilon(K_X+\Delta+\epsilon_0 H).$$
	This implies that $\frac{1}{m_0}\mathrm{div}(f_0)=\frac{1}{m_\epsilon}\mathrm{div}(f_\epsilon)$. Therefore taking the limit for $\epsilon\rightarrow 0$, we have $\frac{1}{m_0}\mathrm{div}(f_0)+(K_X+\Delta)\geq 0$, which implies $\kappa(X,K_X+\Delta)\geq 0$.
	
	Next we consider when $k\geq 1$. Choose $t_0\ll 1$, let $h:X'\rightarrow Y$ be a projective morphism of smooth projective varieties such that the restriction $h_F$ over the generic fiber $F$ of $h$ is birational to the Iitaka fibration of $K_X+\Delta+t_0H$. By Lemma \ref{thm2}, we have $h_F$ is birational to the Iitaka fibration of $K_X+\Delta+tH$ for any $0<t\ll1$.
	
	Let $\Delta' :=\mu_* ^{-1}\Delta$, $H':=\mu_*^{-1}H$, $E$ be the exceptional divisor of $\mu$. Since $(X,\Delta)$ is klt, for $0\leq t\ll 1$, we have 
	$$ \mu_*\mathcal{O}_{X'}(m(K_{X'}+\Delta' +tH')) \cong \mathcal{O}_X(m(K_X+\Delta+tH)),\ \ \forall m\in \mathbb{N},$$
	Therefore for $t\geq 0$, $\kappa(X,K_X+\Delta+tH)\geq 0$ if and only if $\kappa(X',K_{X'}+\Delta'+tH')\geq 0$. Thus, to prove $\kappa(X,K_X+\Delta)\geq 0$, we only need to prove $\kappa(X',K_{X'}+\Delta')\geq 0$. So we may replace $X$ by $X'$(respectively $\Delta$ by $\Delta'$, $H$ by $H'$) and assume that there is an equidimensional algebraic fibration $h:X\rightarrow Y$.
	
	Let $F$ denote the generic fiber of $h$. By properties of the Iitaka fibration, we have that for all $0<t\ll1$,
	$$ \kappa(F,K_F+\Delta_F+tH_F)=0.$$
	This means on the generic fiber $F$, $ K_F+\Delta_F$ is the limit of effective divisors, therefore it is pseudo effective. Since we assume Conjecture \ref{conj2} holds in dimension $\leq n-1$, then $ \kappa(F,K_F+\Delta_F)\geq 0$, and it is easy to see that
	$$\kappa(F,K_F+\Delta_F)=0.$$
	
	Let $m$ be sufficiently divisible, such that $m(K_X+\Delta)$ is Cartier, since $h$ is equidimensional and $ \kappa(F,K_F+\Delta_F)=0$, $h_*\mathcal{O}_X(m(K_X+\Delta))$ is a reflexive sheaf of rank 1. Also because $Y$ is smooth, $h_*\mathcal{O}_X(m(K_X+\Delta))$ is an invertible sheaf. So there is a Cartier divisor $D'$ on $Y$, such that 
	\begin{equation}
	\label{nv}
	h_*\mathcal{O}_{X}(m(K_X+\Delta))=\mathcal{O}_Y(D').
	\end{equation}
	Let $A'$ be a sufficiently ample divisor on $Y$, such that $D'+A'$ is ample on $Y$. Therefore $h_*\mathcal{O}_{X}(m(K_X+\Delta)+h^*A)=\mathcal{O}_Y(D'+A')$ is ample on $Y$ and $m(K_X+\Delta)+h^*A'$ is big, let $A:=\frac{1}{m}A'$. By Remark \ref{piecewise linear}, for $0\leq t\ll 1$, we can find $D'_t$ and $F_t$ such that 
	$$K_X+\Delta+h^*A+tH\sim_{\mathbb{Q}} h^*D'_t+F_t$$
	where $D'_t$ is a linear function of $t$ and $\mathcal{O}_Y\cong h_*\mathcal{O}_X(\lfloor iF\rfloor)$ for all $i\geq 0$. Let $D_t:=D'_t-A$, which is also a linear function of $t$, and we have 
	$$K_X+\Delta+tH\sim_{\mathbb{Q}} h^*D_t+F_t.$$
	Since for $t>0$, we have $\kappa(X,K_X+\Delta+tH)= \mathrm{dim}(Y)$, which means $D_t$ is big for every $t>0$. Therefore $D_0$ is pseudo effective. By the Canonical bundle formula, perhaps replacing $h:X\rightarrow Y$ by a higher model, we can find $\mathbb{Q}-$divisors $B$ and $J$ on $Y$ such that
	$$D_0\sim_{\mathbb{Q}} K_Y+B+J.$$
	where $(Y,B)$ is klt pair and $J$ is nef. 
	
	If Conjecture \ref{conj1} holds in dimension$\leq n-1$, $J$ is semiample, choose a general member $C\in |J+B|_\mathbb{Q}$, then $(Y,C)$ is klt, and $ K_Y+C$ is pseudo effective, by Conjecture \ref{conj2} in dimension $\leq n-1$, $\kappa(X,K_Y+C)\geq 0$. Therefore $\kappa(X,K_X+\Delta)\geq 0$.
\end{proof}

\subsection{MMP with scaling}
\begin{defn}(MMP with scaling)
	Let $(X_1,\Delta_1)$ and $(X_1,\Delta_1+H_1)$ be two klt pairs such that $K_{X_1}+\Delta_1+H_1$ is nef, $\Delta_1\geq 0$, and $H_1$ is $\mathbb{Q}-$Cartier and pseudo effective. Suppose that either $K_{X_1}+\Delta_1$ is nef or there is an extremal ray $R_1$ such that $(K_{X_1}+\Delta_1).R_1 <0$ and $(K_{X_1}+\Delta_1+\lambda_1 H_1).R_1=0$ where
	$$\lambda_1:=\mathrm{inf}\{t\geq 0\ |\ K_{X_1}+\Delta_1 +tH_1 \mathrm{is\ nef}\}$$
	Now, if $K_{X_1}+\Delta_1$ is nef or if $R_1$ defines a Mori fibre structure, we stop. Otherwise assume that $R_1$ gives a divisorial contraction or a log flip $X_1\dashrightarrow X_2$. We can now consider $(X_2,\Delta_2+\lambda_1 H_2)$ where $\Delta_2+\lambda_1H_2$ is the birational transform of $ \Delta_1+\lambda_1H_1$ and continue. That is, suppose that either $K_{X_2}+\Delta_2$ is nef or there is an extremal ray $R_2$ such that $(K_{X_2}+\Delta_2).R_2<0$ and $ (K_{X_2}+\Delta_2+\lambda_2H_2).R_2=0 $ where 
	$$\lambda_2:=\mathrm{inf}\{t\geq 0\ |\ K_{X_2}+\Delta_2 +tH_2 \mathrm{is\ nef}\}$$
	By continuing this process, we obtain a sequence of numbers $\lambda_i$ and a special kind of MMP which is called the MMP on $K_{X_1}+\Delta_1$ with scaling of $H_1$. Note that by definition $\lambda_i \geq \lambda_{i+1} $ for every $i$.
\end{defn}
\begin{proof}[Proof of Theorem \ref{MMP with scaling}]
	Because $K_X+\Delta+ H$ has a good minimal model, by Theorem \ref{vandk}, $\nu(X,K_X+\Delta+H)=\kappa(X,K_X+\Delta+H)=k$ for some nonnegative integer $k$. Since $H$ is pseudo effective, by Proposition \ref{prop}. (4), $\nu(X,K_X+\Delta +tH)\leq k,\ \forall t\in [0,1]$. By assumption, $\kappa(X,K_X+\Delta)\geq 0$, therefore we have 
	$$ k\geq \nu(X,K_X+\Delta+tH)\geq \kappa(X,K_X+\Delta+tH )\geq  \mathrm{max}\{\kappa(X,K_X+\Delta),\ \kappa(X,K_X+\Delta+H)\}\geq k,$$
	for every $t\in (0,1]$. Then $(X,K_X+\Delta+tH)$ has a good minimal model for any $ t\in (0,1]$.
	
	Suppose $\lambda>0$. By Theorem \ref{main}, there exists $\epsilon>0$, such that for every $t\in[\lambda,\lambda+\epsilon]$, $K_X+\Delta+tH$ has the same ample model $Z$. Consider the interval $ I:=[\lambda,\lambda+\epsilon]$. By definition of MMP with scaling, we have infinitely many birational contractions $\phi _i$, such that $\phi_i$ is a minimal model for $K_X+\Delta+tH,t\in[\lambda_i,\lambda_{i+1}]$, and $K_{X_i}+\phi_{i*}\Delta+t\phi_{i*}H$ is not nef if $t>\lambda_{i+1}$. 
	
	On the other hand, by Theorem \ref{main2}, there are finitely many birational contraction $f_j:X\dashrightarrow Y_j,1\leq j\leq m$ and we can divide $I$ into finitely many closed interval $I= \cup _{1\leq j\leq m}[t_j,t_{j+1}]$, such that for $ t\in(t_j,t_{j+1})$, $f_j$ is a good minimal model for $K_X+\Delta+tH$. It is easy to see that there exist $i,j$, such that $ t_j<\lambda_i<\lambda_{i+1}<t_{j+1}$.
	
	Consider two rational number $r_1,r_2\in (\lambda_i,\lambda_{i+1})$. By assumption, $ \phi_i$ and $f_j$ are respectively good minimal model of $K_X+\Delta +r_1H$ and $K_X+\Delta+r_2H$. Let $ h: X_i\rightarrow Z$ and $ g:Y_j\rightarrow Z$ be the morphism from good minimal model to ample model, then there are two divisors $D_1,D_2$ on $Z$, such that $ K_{X_i}+\phi_{i*}\Delta+r_k\phi_{i*} H \sim_{\mathbb{Q}} h^*D_k$ and $ K_{Y_j}+f_{j*}\Delta +r_k f_{j*} H\sim_{\mathbb{Q}} g^*D_k$ for $k=1,2$. Therefore by linearity, for all $t\in[0,1]$, $( X_i,\phi_{i*}\Delta+t\phi_{i*} H)$ is crepant birational with $(Y_j, f_{j*}\Delta +t f_{j*} H)$, which means $\phi_i$ is also a minimal model for $ K_X+\Delta +tH, t\in [t_j,t_{j+1}]$, this contradicts with the definition of MMP with scaling. 
\end{proof}

\end{document}